\documentclass[11pt,a4paper]{amsart}
\usepackage{amscd,amssymb,amsmath,latexsym,graphicx,epsfig,color,url}
\usepackage[active]{srcltx}
\usepackage[all]{xy}
\usepackage{hyperref}
\hypersetup{breaklinks=true}
\newcommand{\thickhline}{%
    \noalign {\ifnum 0=`}\fi \hrule height 1pt
      \futurelet \reserved@a \@xhline
}

\newtheorem{theorem}{Theorem}[section]
\newtheorem{corollary}[theorem]{Corollary}
\newtheorem{lemma}[theorem]{Lemma}
\newtheorem{proposition}[theorem]{Proposition}

\theoremstyle{definition}

\newtheorem{example}[theorem]{Example}
\newtheorem{remark}[theorem]{Remark}

\newcommand{\FF}{\mathbb F}
\newcommand{\KK}{\mathbb K}

\DeclareMathOperator{\diag}{diag}

\DeclareMathOperator{\Inv}{Inv}     
\DeclareMathOperator{\Chinv}{Chinv} 
\DeclareMathOperator{\Hinv}{Hinv}   
\DeclareMathOperator{\End}{End}   

\begin{document}

\title{The lattice of characteristic subspaces of an endomorphism with Jordan-Chevalley decomposition}
\author{David Mingueza}
\address{Accenture, Passeig Sant Gervasi 51-53, 08022 Barcelona, Spain}
\email{david.mingueza@outlook.es}
\author{M.~Eul\`{a}lia Montoro}

\address{Departamento de Matem\'aticas e Inform\'atica. Universitat de Barcelona, \\ Gran Via de les Corts Catalanes 585,
08007 Barcelona, Spain}
\email{eula.montoro@ub.edu}
\author{Alicia Roca}
\address{Departamento de Matem\'{a}tica Aplicada, IMM, Universitat Polit\`ecnica de Val\`{e}ncia,\\ Camino de Vera s/n, 46022 Val\`encia, Spain}
\email{aroca@mat.upv.es}

\thanks{Montoro is supported by the Spanish MINECO/FEDER research project MTM 2015-65361-P, Roca is supported by grant MTM2017-83624-P MINECO }
\date{\today}


\begin{abstract}
Given an endomorphism $A$ over a finite dimensional vector space having Jordan-Chevalley decomposition, the lattices of invariant and hyperinvariant subspaces of $A$ can be obtained from 
the nilpotent part of this decomposition. 
We extend this result for lattices of characteristic subspaces. We also obtain a generalization of Shoda's Theorem about the characterization of the existence of
characteristic non hyperinvariant subspaces. 
\end{abstract}
\maketitle

\section{Introduction}\label{sec:Intro}

\medskip

The lattice of characteristic subspaces of an endomorphism over a finite dimensional space has been studied in \cite{Ast1, Ast2, MMP13, MMR16}, where structural properties of the lattice
have been given  when the minimal polynomial of the endomorphism splits over the underlying field $\mathbb{F}$. It was proved in (\cite{Ast1})
that only if $\FF=GF(2)$, the lattices of characteristic and hyperinvariant subspaces may not coincide. When the minimal polynomial of the 
endomorphism does not split over  $\mathbb{F}$, the lattice of characteristic subspaces has not been described. 
The aim of this paper is to analyze this case when the minimal polynomial of the endomorphism is separable. 

The results are based on the  Jordan-Chevalley decomposition of an endomorphism which exists if the minimal polynomial 
of the endomorphism is separable (see \cite{HoffKun71}). In particular, on perfect fields every irreducible polynomial is separable, 
therefore the Jordan-Chevalley decomposition holds.

\medskip

The paper is organized as follows: in Section \ref{sec:preli} we introduce some definitions and previous results. We present 
some lemmas showing that the study the lattices of the invariant and hyperinvariant subspaces of an endomorphism can be reduced 
to the case of an endomorphism where the minimal polynomial is a power of an irreducible polynomial $p$.

In Section \ref{The characteristic case}, out of the Jordan-Chevalley decomposition of an endomorphism into its commuting semisimple and nilpotent 
parts, we reduce the problem to the study of the lattice of characteristic subspaces of the associated nilpotent part.
Finally, we conclude that characteristic non hyperinvariant subspaces can only appear when the irreducible factor of 
the minimal polynomial is of degree 1, in other words, if $p$ splits over $\mathbb{F}$. This result extends Shoda's theorem (Theorema~\ref{Shodate}) for the separable case.

\section{Preliminaries}\label{sec:preli}

We introduce some  definitions and previous results which will be used throughout the paper.

\medskip

Let $\FF$ be a field. Let $V$  be an $n$-dimensional vector space over $\mathbb{F}$ and  $f : V \rightarrow V$ an  
endomorphism. We denote by $A\in M_{n}(\mathbb{F})$ its associated matrix with respect to given basis and  by $m_A$ the minimal polynomial of $A$. In what follows we will identify 
$f$ with $A$. The degree of a polynomial $p$ is written as $\deg(p)$, and $|\cdot|$ is the ``cardinality of''.

\medskip

The lattice of the vector subspaces of $V$ will be denoted by $\L_{\mathbb{F}}(V)$ .

A subspace $Y\subseteq V$ is 
\textit{invariant} with respect to $A$ if $AY\subseteq Y$.
We denote by $\Inv_{\mathbb{F}}(A)$ the lattice of invariant subspaces.

The centralizer of $A$ over $\mathbb{F}$ is the algebra
$Z_{\mathbb{F}}(A)=\{ B\in M_{n}(\mathbb{F}): AB = BA\}$. If we only take those endomorphisms $B$ that are automorphisms we write $Z^{*}_{\mathbb{F}}(A)=\{B\in M_{n}(\mathbb{F}): AB = BA, \ \det(B)\neq0\}$.

A subspace $Y\subseteq V$ is called \textit{hyperinvariant} with respect to $A$ if $BY\subseteq Y$ for all matrices $B\in Z_{\mathbb{F}}(A)$. We denote by $\Hinv_{\mathbb{F}}(A)$ the lattice 
of hyperinvariant subspaces.

An invariant subspace $Y\subseteq V$ with respect to $A$ is called \textit{characteristic} if $BY\subseteq Y$ for all matrices $B\in Z^{*}_{\mathbb{F}}(A)$ and $AY\subseteq Y$. We denote by $\Chinv_{\mathbb{F}}(A)$ the lattice 
of characteristic subspaces.

Obviously,
$$\Hinv_{\mathbb{F}}(A) \subseteq  \Chinv_{\mathbb{F}}(A) \subseteq \Inv_{\mathbb{F}}(A).$$

\medskip

The lattices of invariant and hyperinvariant subspaces allow the following decompositions.

\begin{proposition} \cite{BrickFill65} \label{invdecomp}
Let $A$ and $B$ be endomorphisms on finite dimensional vector spaces $V$ and $W$ respectively, over a field $\FF$. The following properties 
are equivalent:
\begin{enumerate}
\item The minimum polynomials of $A$ and $B$ are relatively prime.
\item
$\Inv_{\mathbb{F}}(A \oplus B) = \Inv_{\mathbb{F}}(A) \oplus \Inv_{\mathbb{F}}(B)$.
 \end{enumerate}
\end{proposition}

\begin{proposition} \cite{FillHL77}\label{hinvdecomp}
Let $A$ and $B$ be endomorphisms on finite dimensional vector spaces $V$ and $W$ respectively, over a field $\FF$. The following properties
are equivalent:
\begin{enumerate}
\item The minimum polynomials of $A$ and $B$ are relatively prime.
\item
$\Hinv_{\mathbb{F}}(A \oplus B) = \Hinv_{\mathbb{F}}(A) \oplus \Hinv_{\mathbb{F}}(B)$.

 \end{enumerate}
\end{proposition}

As a consequence, if $m_A=p_1^{k_1} \cdots p_r^{k_r}$ is the prime decomposition of $m_A$ for a given $A\in M_{n}(\mathbb{F})$, taking into account the primary decomposition 
$$V=V_1\oplus\cdots \oplus V_r,  \quad V_i=\ker(p_i(A)^{k_i}), \quad i=1, \ldots, r,$$
and $A_i=A|V_i$, \ $i=1, \ldots r$, the next result is satisfied (see \cite{BrickFill65, FillHL77}).

\begin{proposition}
Let $A\in M_{n}(\mathbb{F})$ and $m_A=p_1^{k_1} \cdots p_r^{k_r}$ with $p_1, \ldots, p_r$ distinct irreducible polynomials. Then,
\begin{enumerate}
\item
$\Inv_{\mathbb{F}}(A) = \Inv_{\mathbb{F}}(A_1)  \oplus \cdots \oplus \Inv_{\mathbb{F}}(A_r).$
\item
$\Hinv_{\mathbb{F}}(A) = \Hinv_{\mathbb{F}}(A_1)  \oplus \cdots \oplus \Hinv_{\mathbb{F}}(A_r).$
\end{enumerate}

\end{proposition}

If $p(x)\in \FF[x]$ is an irreducible polynomial and  $\alpha \notin \mathbb{F}$
is a root of $p(x)$, we denote by $\FF(\alpha)$ the minimal extension of $\mathbb{F}$ containing  $\alpha$. The next result can be found 
in \cite{BrickFill65}. Although a proof of the following property 4 was included in  \cite{BrickFill65}, we give here a simple proof of the 
whole lemma.

\begin{lemma} \cite{BrickFill65} \label{mainlema}
Let $A\in M_{n}(\mathbb{F})$. Assume that
$m_A=p_{0}+p_{1}x+\ldots+p_{s}x^{s} \in \FF[x]$ is  irreducible, and that  $\alpha$ is a root of $p$ such that $\alpha \notin \mathbb{F}$.
 Let $\mathbb{K}=\{a_{0}I_{n}+a_{1}A+\ldots+a_{s-1}A^{s-1}\,|a_{i}\in \mathbb{F}\}$ be the algebra of polynomials of degree at most $s$. Then,
 \begin{enumerate}
  \item $\mathbb{K}$ is a field isomorphic to $\mathbb{F}(\alpha)$.
  \item $V$ is a $\mathbb{K}-$vector space.
  \item $A$ is $\mathbb{K}-$linear.
  \item $\Inv_{\mathbb{F}}(A)=\L_{\mathbb{K}}(V)$.
 \end{enumerate}
\end{lemma}

\begin{proof}
 \begin{enumerate}
  \item The following application
  $$\Gamma:\mathbb{K}\longrightarrow\mathbb{F}(\alpha),$$
 with $\Gamma(a_{0}I_{n}+a_{1}A\ldots+a_{s-1}A^{s-1})=a_{0}+a_{1}\alpha+\ldots+a_{s-1}\alpha^{s-1}$ is an isomorphism. 
  \item For all $\lambda=a_{0}I_{n}+a_{1}A+\ldots+a_{s-1}A^{s-1}\in\mathbb{K}$, it is satisfied that
  $$\lambda v=(a_{0}I_{n}+a_{1}A+\ldots+a_{s-1}A^{s-1})v\in V,\quad \forall v\in V.$$
  \item For all $\lambda \in \mathbb{K}$, $A\lambda v=\lambda Av$.
  \item The following equivalences hold, 
  $$ W\in \Inv_{\mathbb{F}}(A)\Leftrightarrow AW\subseteq W\Leftrightarrow  \lambda W\subseteq W,\,\forall \lambda\in\mathbb{K} \Leftrightarrow W\in \L_{\mathbb{K}}(V).$$
 \end{enumerate}
\end{proof}

Notice that properties $1-4$ of this lemma are trivially true if $s=1$.
\medskip

A polynomial is called \textit{$p$-primary} if it is of the form $p^r$ for some irreducible polynomial $p$ and a positive integer $r$.
A polynomial is called \textit{separable} if it is coprime with its derivate (\cite{Lang02}).

A linear operator S on a finite-dimensional vector space is \textit{semisimple} if every S-invariant subspace has a complementary S-invariant subspace. A linear operator N on a finite-dimensional 
vector space is \textit{nilpotent} if $N^k = 0$ for some positive integer k.

\medskip
The next lemma contains the Jordan-Chevalley decomposition of a matrix, which plays a key role in this work. The result 
can be found in \cite{HoffKun71}.

\begin{lemma}\label{sn}(Jordan-Chevalley decomposition)
Let $A\in M_{n}(\mathbb{F})$ be a matrix  with $m_{A}=p^r$, where $p$ is irreducible and separable. Then,  there is a unique decomposition 
$$A=S+N,$$ where $S$ is  semisimple, $N$ is nilpotent and $SN=NS$. Moreover,  $S$ and $N$ are polynomials in $A$. 
\end{lemma}

\begin{remark} \cite{BrickFill65}
 
If $A$ has Jordan-Chevalley decomposition then there exist a transformation into a rational canonical form such that $A$ is similar to: 
 \begin{equation*}\label{Jdecomp}
 \diag(C_1, C_2, \ldots, C_m),
 \end{equation*}
 with
$$C_{i}=\left(\begin{array}{ccccc}C& 0 &\ldots&0&0\\
                                   I_{s} &C&\ldots&0&0\\
                                   \vdots&\vdots&\ddots&\vdots & \vdots\\
                                    0&0&\ldots &C&0\\
                                   0&0& \ldots&I_{s}&C
                 \end{array}\right)\in M_{n_{i}}(\mathbb{F}), \quad i=1,2, \ldots, m,$$
where $C$ is the companion matrix of $p$, and
$n_{1}\geq n_{2}\geq \dots \geq n_{m}>0$,  $n_{1}+\ldots+n_{m}=n$,  $n_{1}=sr$ (where $s=\deg(p)$), then
$$S=\left(\begin{array}{ccccc}C& 0 &\ldots&0&0\\
                                   0&C&\ldots&0&0\\
                                   \vdots&\vdots&\ddots&\vdots & \vdots\\
                                    0&0&\ldots &C&0\\
                                   0&0& \ldots&0&C
                 \end{array}\right),
                 \quad
              N=   \left(\begin{array}{ccccc}0& 0 &\ldots&0&0\\
                                   I_{s} &0&\ldots&0&0\\
                                   \vdots&\vdots&\ddots&\vdots & \vdots\\
                                    0&0&\ldots &0&0\\
                                   0&0& \ldots&I_{s}&0
                 \end{array}\right).$$

\end{remark}

Taking advantage of the Jordan-Chevalley decomposition it can be proved that the lattices of the invariant and hyperinvariant subspaces of $A$ over $\mathbb{F}$ can be obtained as
lattices of the corresponding subspaces  of a nilpotent matrix over a different field $\mathbb{K}$ (see \cite{BrickFill65} and \cite{FillHL77}, respectively). The results are included in the next theorem.
Although proofs are sketched in~\cite{BrickFill65,FillHL77}, we provided them in detail.

\begin{theorem}  \label{inv-hyperinv}
Let $A\in M_{n}(\mathbb{F})$. Assume that $m_A$ is $p-$primary $(m_{A}=p^{r})$ with $p$ separable and $\deg (p)=s$. Let $A=S+N$ be the Jordan-Chevalley decomposition of $A$. Let $\mathbb{K}=\{a_{0}I_{n}+a_{1}S+\ldots+a_{s-1}S^{s-1}\,|a_{i}\in \mathbb{F}\}$.
 Then,
 \begin{enumerate}
\item $\Inv_{\mathbb{F}}(A)=\Inv_{\mathbb{K}}(A)=\Inv_{\mathbb{K}}(N).$
\item $\Hinv_{\mathbb{F}}(A)=\Hinv_{\mathbb{K}}(A)=\Hinv_{\mathbb{K}}(N).$
\end{enumerate}
\end{theorem}

\begin{proof}
\begin{enumerate}
\item $W\in \Inv_{\mathbb{F}}(A)\Leftrightarrow W\in\Inv_{\mathbb{K}}(A)$ (because $A$ is $\mathbb{K}-$ linear).
Moreover, since $S\in \mathbb{K}$
$$ \Inv_{\mathbb{K}}(A)=\Inv_{\mathbb{K}}(A-S)=\Inv_{\mathbb{K}}(N).$$
 \item Notice that
$$Z_{\mathbb{K}}(A)=Z_{\mathbb{F}}(A)\cap Z_{\mathbb{F}}(S)=Z_{\mathbb{F}}(A),$$
and
$$Z_{\mathbb{K}}(A)=Z_{\mathbb{K}}(A-S)=Z_{\mathbb{K}}(N).$$
Therefore
 $$\Hinv_{\mathbb{F}}(A)=\Hinv_{\mathbb{K}}(A)=\Hinv_{\mathbb{K}}(A-S)=\Hinv_{\mathbb{K}}(N).$$
\end{enumerate}
\end{proof}

Next example illustrates the above lemma for the field of real numbers.

\begin{example}\label{exhyp}
 Let  $\mathbb{F}=\mathbb{R}$, $V=\mathbb{R}^{4}$ and
 $$A=\left(\begin{array}{cccc}0&1&0&0\\-1&0&0&0\\1&0&0&1\\ 0&1&-1&0 \end{array}\right).$$
The minimal polynomial of $A$ is $m_{A}=(x^{2}+1)^{2}$. The Jordan-Chevalley decomposition of $A$ is given by
  $$ S=\left(\begin{array}{cccc}0&1&0&0\\-1&0&0&0\\0&0&0&1\\ 0&0&-1&0 \end{array}\right), \quad
  N=\left(\begin{array}{cccc}0&0&0&0\\0&0&0&0\\1&0&0&0\\ 0&1&0&0 \end{array}\right),$$
and
$\mathbb{K}=\left\{ a_{0}I_{4}+a_{1}S,\,a_{i}\in\mathbb{R}\right\}$. Notice that the Jordan chains of $N$ are

$$e_{1} \rightarrow e_{3}\rightarrow 0,$$

$$e_{2} \rightarrow e_{4}\rightarrow 0.$$

Then,

 $$ \Inv_{\mathbb{R}}(A)=\Inv_{\mathbb{K}}(N)= \left\{0,\langle e_{3},e_{4}\rangle,V \right\}.$$

\medskip

The lattice of hyperinvariant subspaces is
$$\Hinv_{\mathbb{R}}(A)=\Hinv_{\mathbb{K}}(N)=\{0,\langle e_{3},e_{4} \rangle, V\}.$$

 \end{example}

\medskip

Concerning to characteristic subspaces, Shoda's theorem characterizes the existence of characteristic non hyperinvariant subspaces.

\begin{theorem}\cite{Shoda}\label{Shodate}
Let $V$ be a finite-dimensional vector space over the field $\mathbb{F} = GF(2)$ and let $f :V \rightarrow V$ be a
nilpotent linear operator. The following statements are equivalent:
\begin{enumerate}
\item There exists a characteristic subspace of $V$ which is not hyperinvariant.
\item For some numbers $r$ and $s$ with $s > r + 1$ the Jordan form of $f$ contains exactly one Jordan block of size $s$ and exactly one block of size $r$.
 \end{enumerate}
\end{theorem}

Astuti-Wimmer proved (\cite{Ast1}) that characteristic non hyperinvariant subspaces can only exist on the field $GF(2)$.

\begin{theorem}\cite{Ast1}\label{AstWim1}
Let $V$ be a finite dimensional vector space over a field $\mathbb{F} $ and let $f : V \rightarrow V$ be a linear operator. Assume that 
the minimal polynomial of $f$ splits over $\FF$. If $|\FF|>2$, then $\Chinv_{\FF}(f)=\Hinv_{\FF}(f)$.
\end{theorem}

\section{Reduction to the nilpotent case for characteristic subspaces}\label{The characteristic case}

In this section we focus on the study of the lattice of characteristic subspaces of an endomorphism over a field when the irreducible 
factors of the minimal polynomial are separable.

First of all we see that the general case can be reduced to the especific case of endomorphisms having minimal polynomials with an 
unique irreducible factor, as in the lattices of invariant of hyperinvariant subspaces.

\begin{lemma}

Let $A$ and $B$ be endomorphisms on finite dimensional vector spaces $V$ and $W$ respectively, over a field $\FF$. The following properties 
are equivalent:
\begin{enumerate}
\item The minimum polynomials of $A$ and $B$ are relatively prime.
\item
$\Chinv_{\mathbb{F}}(A \oplus B) = \Chinv_{\mathbb{F}}(A) \oplus \Chinv_{\mathbb{F}}(B)$.

 \end{enumerate}

\end{lemma}

\begin{proof} 

First assume that (1) is true. It is evident that $\Chinv(A\oplus B) \subseteq \Chinv(A)\oplus\Chinv(B)$ holds even for an arbitrary
polynomials. 
 
As it is stated in \cite{Supru68}, $Z(A\oplus B)=Z(A)\oplus Z(B)$, therefore every commuting automorphism of $A\oplus B$ is of the 
form $X=\left(\begin{array}{cc}X_{1}&0\\
                                   0&X_{2}\end{array}\right)$ with $X_{1}\in Z^{\ast}(A)$ and $X_{2}\in Z^{\ast}(B)$.

Taking $W=W_{1}\oplus W_{2} \in\Chinv(A)\oplus \Chinv(B)$, clearly $W\in\Chinv(A\oplus B)$ follows.

\medskip

Assume now that (2) is true. 

As $\Chinv(A\oplus B)=\Chinv(A)\oplus\Chinv(B)$, it is satisfied that $V \oplus 0\in \Chinv(A\oplus B)$. Notice that if  $X$ is a matrix such that $XA=BX$, then $\left(\begin{array}{cc}I&0\\
                                   X&I\end{array}\right)\in Z^{\ast}(A\oplus B)$. Therefore, $V\oplus 0$ must be invariant for this matrix and this implies $X=0$.

If $\gcd(m_{A}, m_{B})=q$, $\deg(q)>0$, then $m_{A}=qf$ and $m_{B}=qh$ for some polynomials $f,h$, and w.l.o.g. we can suppose $\gcd(q,f)=1$. With respect to an appropriate basis, we can write
$A=\left(\begin{array}{cc}C_{q}&0\\
                                   0&F\end{array}\right)$ and $B=\left(\begin{array}{cc}C_{q}&H\\
                                   0&G\end{array}\right)$ where $C_{q}$ is the companion matrix of $q$ and for some matrices $F,H$ and $G$.
                                   
                                   Let $U\in Z^{\ast}(C_{q})$, then 
                                   $X=\left(\begin{array}{cc}U&0\\
                                   0&0\end{array}\right)\neq 0$ satisfies $XA=BX$, which is clearly a contradiction.

\end{proof} 

\begin{corollary}
 
Let $A\in M_{n}(\mathbb{F})$ and $m_{A}=p_{1}^{k_{1}}\ldots p_{r}^{k_{r}}$ with $p_{i}$ distinct irreducible polynomials. Then,
\begin{equation} \Chinv(A)=\Chinv(A_{1})\oplus\dots\oplus\Chinv(A_{r}),
\end{equation}
where $A_{i}\in M_{n_{i}}(\mathbb{F})$ is the restriction of $A\in M_{n}(\mathbb{F})$ to $V_{i}=\ker((p_{i}(A))^{k_{i}})$, $i=1, 2, \ldots, r$, $(V=V_{1}\oplus\dots\oplus V_{r})$.
\end{corollary}

In the next lemma we see that the lattice of characteristic subspaces of $A$ can also be determined on a different field for a nilpotent matrix, using Jordan-Chevalley decomposition.

\begin{lemma}\label{charlema}
Let $A\in M_{n}(\mathbb{F})$. Assume that $m_A$ is $p-$primary ($m_{A}=p^{r}$) with $p$ separable and $\deg (p)=s$. Let $A=S+N$ be 
the Jordan-Chevalley decomposition of $A$. Let $\mathbb{K}=\{a_{0}I_{n}+a_{1}S+\ldots+a_{s-1}S^{s-1}\,|a_{i}\in \mathbb{F}\}$.
Then,
\begin{enumerate}

  \item $\Chinv_{\mathbb{F}}(A)=\Chinv_{\mathbb{K}}(A)=\Chinv_{\mathbb{K}}(N).$
\end{enumerate}
\end{lemma}

\begin{proof}

Similarly to the proof of Theorem \ref{inv-hyperinv}, 
$$Z_{\mathbb{K}}^*(A)=Z_{\mathbb{F}}^*(A)\cap Z_{\mathbb{F}}^*(S)=Z_{\mathbb{F}}^*(A),$$
and
$$Z_{\mathbb{K}}^*(A)=Z_{\mathbb{K}}^*(A-S)=Z_{\mathbb{K}}^*(N).$$
Therefore
 $$\Chinv_{\mathbb{F}}(A)=\Chinv_{\mathbb{K}}(A)=\Chinv_{\mathbb{K}}(A-S)=\Chinv_{\mathbb{K}}(N).$$
\end{proof}

Finally we obtain the conclusion that characteristic non hyperinvariant subspaces can only exist over Jordan blocks.

\begin{theorem}
Let $A$ be an endomorphism on a finite dimensional space $V$ over a field $\mathbb{F}$. Assume that $m_A$ is $p-$primary $(m_{A}=p^{r})$ with $p$ separable. 
If $\Chinv_{\mathbb{F}}(A)\setminus \Hinv_{\mathbb{F}}(A)\neq \{\emptyset\}$, then $\deg(p)=1$ and $\mathbb{F}=GF(2)$.
\end{theorem}
\begin{proof}
Let us assume that $\Chinv_{\mathbb{F}}(A)\setminus \Hinv_{\mathbb{F}}(A)\neq \{\emptyset\}$ and $\deg(p)>1$. As $p$ is separable, let $A=N+S$ be  the Jordan-Chevalley decomposition of $A$. By 
Lemma \ref{charlema} we know that $\Chinv_{\mathbb{F}}(A)=\Chinv_{\mathbb{K}}(N)$ with $\mathbb{K}$ defined as in Lemma \ref{charlema}. By assumption $\deg(p)>1$, which implies that $|\mathbb{K}|>2$, and 
by \cite[Theorem 3.4]{Ast1}, $\Chinv_{\mathbb{K}}(N)=\Hinv_{\mathbb{K}}(N)$. Finally, by Theorem \ref{inv-hyperinv}, $\Hinv_{\mathbb{K}}(N)=\Hinv_{\mathbb{F}}(A)$. Therefore, 
$\Chinv_{\mathbb{F}}(A)=\Hinv_{\mathbb{F}}(A)$, which is a contradiction.
 \end{proof}

We see that for the existence of characteristic non-hyperinvariant subspaces it is necessary that $\deg(p)=1$ and $\mathbb{F}=GF(2)$. If we take into account the Shoda condition, we obtain a necessary and sufficient condition.
 
\begin{corollary}
Let $A$ be an endomorphism on a finite dimensional space $V$ over a field $\mathbb{F}$. Assume that $m_A$ is $p-$primary $(m_{A}=p^{r})$ with $p$ separable. Then, the next two statement are equivalent:
 \begin{enumerate}
  \item $\Chinv_{\mathbb{F}}(A) \backslash  \Hinv_{\mathbb{F}}(A) \neq \{\emptyset\}$.
  \item  $\mathbb{F}= GF(2)$, $\deg(p)=1$ and the Shoda condition is satisfied.
  \end{enumerate}

\end{corollary}

 \medskip

We show next how the lattice of characteristic subspaces can be obtained in two examples. In the first 
one $\Hinv_{\mathbb{F}}(A) = \Chinv_{\mathbb{F}}(A)$, and in the second one 
$\Chinv_{\mathbb{F}}(A)\not\subseteq\Hinv_{\mathbb{F}}(A)$ for some matrices $A$.
\begin{example}\label{exchar}
 Let  $\mathbb{F}=GF(2)$ and
 $$A=\left(\begin{array}{cccccc|cc}0&1&0&0&0&0&0&0\\1&1&0&0&0&0&0&0\\1&0&0&1&0&0&0&0\\0&1&1&1&0&0&0&0\\0&0&1&0&0&1&0&0\\0&0&0&1&1&1&0&0
 \\
\hline
 0&0&0&0&0&0&0&1 \\0&0&0&0&0&0&1&1 \end{array}\right).$$
The minimal polynomial of $A$ is $p$-primary,  $m_{A}=(x^{2}+x+1)^{3}$, $p=x^{2}+x+1$ separable, and the Jordan-Chevalley decomposition of $A$ is 
  $$ S=\left(\begin{array}{cccccc|cc}0&1&0&0&0&0&0&0\\1&1&0&0&0&0&0&0\\0&0&0&1&0&0&0&0\\0&0&1&1&0&0&0&0\\0&0&0&0&0&1&0&0\\0&0&0&0&1&1&0&0
 \\
\hline
  0&0&0&0&0&0&0&1 \\0&0&0&0&0&0&1&1 \end{array}\right), \quad
  N=\left(\begin{array}{cccccc|cc}0&0&0&0&0&0&0&0\\0&0&0&0&0&0&0&0\\1&0&0&0&0&0&0&0\\0&1&0&0&0&0&0&0\\0&0&1&0&0&0&0&0\\0&0&0&1&0&0&0&0
 \\
\hline
  0&0&0&0&0&0&0&0 \\0&0&0&0&0&0&0&0 \end{array}\right).$$
Let $\mathbb{K}=\left\{ a_{0}I_{8}+a_{1}S,\,a_{i}\in GF(2)\right\}$.
 \medskip

\noindent
The Segre characteristic of $N$ is $(3,3,1,1)$ with Jordan chains
 
 $e_{1} \rightarrow e_{3}\rightarrow e_{5} \rightarrow 0$

 $e_{2} \rightarrow e_{4}\rightarrow e_{6} \rightarrow 0$

 $e_{7} \rightarrow 0$

 $e_{8} \rightarrow 0$

\noindent
The characteristic and hyperinvariant  subspaces of $N$ are (see \cite{MMP13}):
$$\Hinv_{\mathbb{K}}(N)=\Chinv_{\mathbb{K}}(N)=
\left\{0,\langle e_{3},e_{5},e_{4},e_{6},e_{7},e_{8}\rangle,\langle e_{3},e_{5},e_{4},e_{6}\rangle,\langle e_{5},e_{6},e_{7},e_{8}\rangle,\langle e_{5},e_{6}\rangle,V \right\}.$$
  \end{example}
 \medskip

 \begin{example}\label{exchar2}
 Let  $\mathbb{F}=GF(2)$ and
 $$A=\left(\begin{array}{cccc}1&0&0&0\\1&1&0&0\\0&1&1&0\\0&0&0&1\\\end{array}\right).$$
 The minimal polynomial of $A$ is $m_{A}=(x+1)^{3}$, $p=x+1$ trivially splits over $\mathbb{F}$ and
  $$ S=\left(\begin{array}{cccc}1&0&0&0\\0&1&0&0\\0&0&1&0\\0&0&0&1\\\end{array}\right), \quad
  N=\left(\begin{array}{cccc}0&0&0&0\\1&0&0&0\\0&1&0&0\\0&0&0&0\\\end{array}\right)$$
Observe that $\mathbb{K}=\mathbb{F}=GF(2)$ as $\deg(p)=1$.

\medskip

\noindent
The Segre characteristic of $N$ is $(3,1)$ with vector chains:

 $e_{1} \rightarrow e_{2}\rightarrow e_{3} \rightarrow 0$

  $e_{4} \rightarrow 0$

\noindent
The hyperinvariant and characteristic non hyperinvariant subspaces, according to \cite{MMP13}, are:
$$\Hinv_{\mathbb{K}}(N)=\left\{ 0,\langle e_{2},e_{3},e_{4} \rangle, \langle e_{2}, e_{3} \rangle,\langle e_{3},e_{4}\rangle,\langle e_{3}\rangle ,V\right\},$$

$$\Chinv_{\mathbb{K}}(N) \backslash \Hinv_{\mathbb{K}}(N) =\left\{e_{2}+e_{4}+\langle e_{3}\rangle \right\}.$$

 \end{example}

\section*{Acknowledgments}

The second author is partially supported by grant MTM2015-65361-P 
MINECO/FEDER, UE. The third author is partially supported by grant
MTM2017-83624-P MINECO.

\bibliographystyle{model1a-num-names}

\end{document}